\documentclass{article}
\usepackage{amsmath,amssymb}
\usepackage{amsthm}
\usepackage{latexsym}
\usepackage{CJK}
\usepackage{graphicx}
\usepackage{subfigure}
\usepackage[pdfstartview=FitH,%
 bookmarksnumbered=true,bookmarksopen=true,%
colorlinks=true,pdfborder=001,citecolor=blue,
linkcolor=blue,urlcolor=blue]{hyperref}

\begin{document}
    \begin{CJK*}{GBK}{song}
    \title{Describe the Fibonacci map by principal nests }
    \author{Hanlin Liu, Mengru Zhang}
    \date{2013/10/12}
    \maketitle
    \begin{abstract}
    In this paper,we studied some properties of the Fibonacci map and find a property of this map related to the notation of principal nest. And we proved this property with some additional conditions equals the original definition of Fibonacci map. Therefore give us a new method to describe Fibonacci map.  \newline
    \end{abstract}
    \section{Introduction}
    A well-known result in one-dynamic system is that an interval map with large critical order and Fibonacci combinatorics has wild attractor. The proof can be found in [2].\par
    A unimodal map \(f:[0,1]\rightarrow [0,1]\) is a continuous map having a unique critical point \(c\), such that  \(f\) is increasing to the left and decreasing to the right side of \(c\). Let \(c_n=f^n(c)\) be the \(n\)-th image of the critical point. We are going to prove the Main Theorem in the following passage:\newline

    \newtheorem*{axiom}{Main theorem}
    \begin{axiom}
    If \(f:[0,1]\rightarrow [0,1]\) is a Fibonacci map, \(q\) is the reverse fixed point of \(f\). \(I^1=(\overset{\frown }{q},q)\), \(\left\{I^k\right\}\) is the principal nest of \(f\) starting from \(I^1\). Then \(f\) satisfies the below properties:\newline
     (1) \(I^k\cap D_{I^k}\) have specifically two components that intersect with \(orb^+(c)\), donated by \(I_0^k\) and \(I_1^k\) in which \(I_0^k=I^{k+1}\)\newline
     (2) \(R_{I^k}|_{I_0^k}=f^{S_{k+1}}\), \(R_{I^k}|_{I_1^k}=f^{S_k}\)\newline
     (3) \(f^{S_k}(c)\notin I^k\) for every \(k\)\newline
     In reverse, if a unimodal map \(f:[0,1]\rightarrow [0,1]\) satisfies the above 3 properties, then \(f\) is Fibonacci map.
    \end{axiom}

    \section{Combinatorial properties of the Fibonacci map}
    We are going to list some properties of the Fibonacci map, and before this, we want to clarify some notations.\par
    Let \(f:[0,1]\rightarrow [0,1]\) be a unimodal map with \(f(0)=f(1)=0\) and critical point \(c\in (0,1)\). For
\(i\geq 0\) and \(x\in [0,1]\), let \(x_i=f^i(x)\) and choose \(x_{-i}\in f^{-i}(x)\) such that \(x_{-i}\) is the nearest point to \(c\).
    In this paper's discussion, \(f\) is symmetric over \(c\), we donate \(\overset{\frown }{x}\neq x\) as the symmetric point of \(x\). The forward orbit of critical point is \(orb^+(c)=\left\{f^k(c);k\geq 0\right\}\)\par

    Let \(S_0=1\) and define \(S_i\) inductively by:\newline
    \(S_i=\min \left\{k\geq S_{i-1}:c_{-k}\in (c_{-S_{i-1}},\overset{\frown }{c}_{-S_{i-1}})\right\}\)\newline
     If \(S_i\) coincides with the Fibonacci numbers: \(S_0=1\), \(S_1=2\) and \(S_{k+1}=S_k+S_{k-1}\), then \(f\) is called a Fibonacci map. Donate \(z_k\) as the nearest point to \(c\) in the set \(f^{-S_k}(c)\).\par
    With the definition of Fibonacci map, we can describe the ordering of the critical point's forward orbit using a notation of Fibonacci sum: \(m=S_{k_1}+S_{k_2}+\text{...}\) where \(k_{i+1}\geq {k_i}+2\) and \(S_{k_i}\) is the Fibonacci number, then we say \(m\) have an expression of Fibonacci sum. It was demonstrated in [2] that every positive integer has a unique expression as Fibonacci sum.\newline

    \newtheorem{lemma}{Lemma}
    \begin{lemma}
    For the Fibonacci map, The forward orbit of the critical point has the following ordering:\newline
    \(c_{S_1+\text{...}}<c_{S_2+\text{...}}<c_{S_5+\text{...}}<\text{...}<c<\text{...}<c_{S_4+\text{...}}<c_{S_3+\text{...}}<c_{S_0+\text{...}}\)\newline
    and \(\left|c_{S_n}\right|>\left|c_{S_{n+1}}\right|\) where \(|x|\)  donates the distance from \(x\) to \(c\). When two Fibonacci sums have the same leading summands \(S_{k_1}+S_{k_2}+\text{...}+S_{k_{i_o}}\) but differ at the \((i_0+1)\)-st summand, write \(S=S_{k_1}+S_{k_2}+\text{...}+S_{k_{i_0}}\), then we have the following ordering:\newline
    \(\left|c_S\right|>\text{...}>|c_{S+S_{k_{i_0}+3}}|>\left|c_{S+S_{k_{i_0}+2}}\right|\) when \(i_0\) is odd. And the inequalities reversed when \(i_0\) is even.
    \end{lemma}

    \begin{proof}
    See in [1].
    \end{proof}

    \newtheorem{corollary}{Corollary}
    \begin{corollary}
    If a Fibonacci sum begin with \(S_n\), then \(c_{S_n+\text{...}}\in (c_{S_n},c_{S_n+S_{n+2}})\)
    \end{corollary}
    \begin{proof}
    It can be derived directly from Lemma 1.
    \end{proof}

    For simplification, let us write \(d_n=f^{S_n}(c)\) and \(y_n=f^{S_n+S_{n+2}}(c)\). Let \(\overset{\frown }{q}\) be the fixed point of \(f\), define \(u_1=\overset{\frown }{q}\) and \(u_{n+1}\in f^{-S_n}(u_n)\) such that \(u_{n+1}\) is nearest to \(c\) and at the same side of \(c\) with \(d_{n+1}\). Define \(U^n=(u_n,\overset{\frown }{u}_n)\)\newline

    \begin{lemma}
    We have the following ordering of \(d_n\), \(y_n\), \(u_n\), \(z_n\):\newline
    \(\left|d_{n+1}\right|<\left|u_n\right|<\left|y_n\right|<\left|z_{n-1}\right|<\left|d_n\right|\)\newline
    \end{lemma}
    \begin{proof}
    See in [2]\newline
    \end{proof}

    Figure 1 below illustrates the relative position of those points:

    \begin{figure}[htbp]
    \centering
    \subfigure[Case when \(d_n\) is on the right side of \(c\)]{
    \includegraphics[width=3.0in]{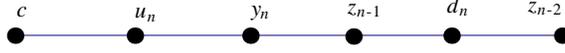}}
    \subfigure[Case when \(d_n\) is on the right side of \(c\)]{
    \includegraphics[width=3.0in]{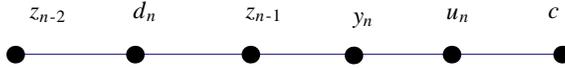}}
    \caption{An illustration of Lemma 2}
    \end{figure}

    \begin{corollary}
    \(f^{S_n}\) is monotone in \((u_n,c)\).
    \end{corollary}

    \begin{proof}
     As \(f^{S_n}\) is monotone in \((z_{n-1},c)\) and \((\overset{\frown }{z}_{n-1},c)\), by Lemma 2, either \((z_{n-1},c)\subset (u_n,c)\) or \((\overset{\frown }{z}_{n-1},c)\subset (u_n,c)\).
     \end{proof}

    \section{Describe the Fibonacci map with principal nest}
    In this section, we are going to prove the Main Theorem proposed in the Introduction part. Before our proof, we need some notations.\par
    An open interval \(J\) is called nice if \(f^n(\partial J)\cap J=\emptyset \) for all \(n\geq 0\). For \(J\subset I\), let \(\left.D(J)=\left\{x\in I:f^k(x)\in J\right. \text{for } \text{some } k\geq 1\right\}\), a component of \(D(J)\) (resp. \(J\cap D(J)\)) is called an entry domain (resp. return domain). And the first entry map \(R_J:D(J)\rightarrow J\) is defined as \(x\rightarrow f^{k(x)}(x)\), where \(k(x)\) is the entry time of \(x\) into \(J\), i.e the minimal positive number satisfies \(f^{k(x)}(x)\in J\).\newline

    \newtheorem{definition}{Definition}
    \begin{definition}
    If \(I^1\) is a nice interval contains \(c\), and \(I^n\) which contains \(c\) is the return domain of \(I^{n-1}\). The sequence \(I^1\supset I^2\supset I^3\supset \text{...}\) is called the principal nest starting from \(I^1\).
    \end{definition}

    \newtheorem{theorem}{Theorem}
    \begin{theorem}
    Let \(f:[0,1]\rightarrow [0,1]\) be a Fibonacci map. \(q\) is the reverse fixed point of \(f\). \(I^1=(\overset{\frown }{q},q)\), \(\left\{I^k\right\}\) is the principal nest of \(f\) starting from \(I^1\). Then \(I^k\cap D_{I^k}\) have specifically two components that intersect with \(orb^+(c)\). Donated by \(I_0^k\) and \(I_1^k\) in which \(I_0^k=I^{k+1}\). And that \(R_{I^k}|_{I_0^k}=f^{S_{k+1}}\), \(R_{I^k}|_{I_1^k}=f^{S_k}\).
    \end{theorem}

    As we already known the relative positions of \(u_n\) but have no idea of \(I^n=(i_n,\overset{\frown }{i}_n)\), where we define \(i_n\) is at the same side of \(c\) with \(d_n\). So our proof is divided into two parts. In the first part, we are going to prove \(U^n\cap D_{U^n}\) have specifically two components that intersect with \(orb^+(c)\), the first return map on the central branch is \(f^{S_{n+1}}\) and \(f^{S_n}\) on the other branch. In the second part, we will discuss the relationship between \(I^n\) and \(U^n\).

    \begin{proof}
    First, we prove the following proposition:\newline
    \newtheorem{prop}{Proposition}
    \begin{prop}
    \(U^n\cap D_{U^n}\) have exactly two components that intersect with \(orb^+(c)\), write as \(U_0^n\) and \(U_1^n\), in which \(c\in U_0^n\) and that \(R_{U^n}|_{U_0^n}=f^{S_{n+1}}\), \(R_{U^n}|_{U_1^n}=f^{S_n}\).
    \end{prop}
    \begin{proof}[Proof of Proposition 1]
    We begin by the case \(n=1\). Remember the properties of Fibonacci map that we have known, we have the following results:\newline
    \(f^{S_1}\) is monotone in \((u_1,c)\), and \(f^{S_2}\) is monotone in \((u_2,c)\),\newline
     \(\left\{
\begin{array}{c}
 f^{S_1}(u_1)=\overset{\frown }{u}_1 \\
 f^{S_2}(u_2)=u_1
\end{array}
\right.\), \(\left\{
\begin{array}{c}
 f^{S_2}(u_2)=q \\
 f^{S_2}(c)=d_2\in U_1\text{  }\text{but }  d_2\notin U_2
\end{array}
\right.\), \(\left\{
\begin{array}{c}
 f^{S_1}(d_2)=d_3 \\
 f^{S_1}(y_2)=d_5
\end{array}
\right.\),\newline
and that \((d_2,y_2)\subset (u_1,u_2)\).\par
From above, we can draw an illustrating figure of the first return map into \(U^1\) as shown in Figure 2. It shows that Proposition 1 is true for \(n=1\).\newline
\begin{figure}[htbp]
\begin{center}
\includegraphics[width=0.755\textwidth]{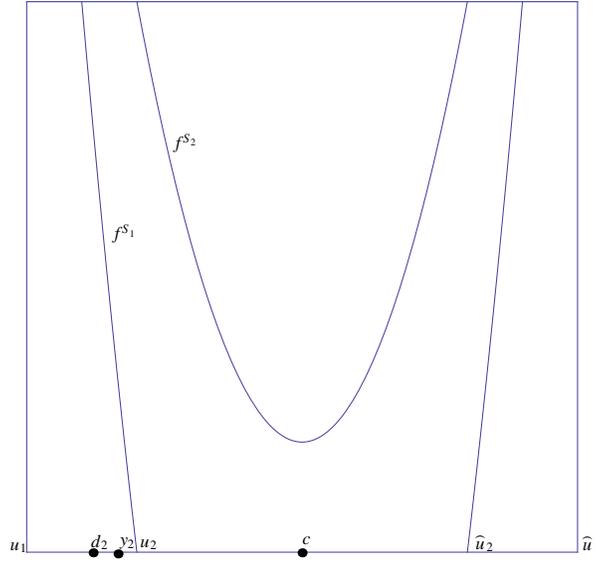}
\caption{First return map into \(U^1\)}
\end{center}
\end{figure}

When it comes to \(n=2\), this time we should focus on \(U^2\), note that:\newline
\(f^{S_2}\) is monotone in \((u_2,c)\), and \(f^{S_3}\) is monotone in \((u_3,c)\), \newline
\(\left\{
\begin{array}{c}
 f^{S_2}(u_2)=q \\
 f^{S_2}(u_3)=u_2
\end{array}
\right.\), \(\left\{
\begin{array}{c}
 f^{S_3}(u_3)=u_1 \\
 f^{S_3}(c)=d_3\in U^2 \text{ but }d_3\notin U^3
\end{array}
\right.\), \(\left\{
\begin{array}{c}
 f^{S_2}(d_3)=d_4\in U^2 \\
 f^{S_2}(y_3)=d_6\in U^2
\end{array}
\right.\).\newline
So \((d_3,y_3)\) is contained in one component of \(U^n\cap D_{U^n}\). Also, we can draw the first return map into \(U^2\) as illustrated in Figure 3.\newline

\begin{figure}[htbp]
\begin{center}
\includegraphics[width=0.755\textwidth]{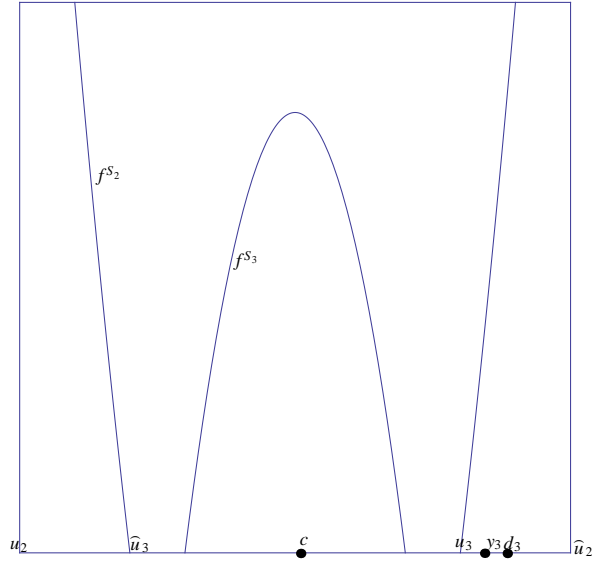}
\caption{First return map into \(U^2\)}
\end{center}
\end{figure}

By similar process, we can draw the first return map into \(U^3\), \(U^4\). The illustrating figure is shown in Figure 4 and Figure 5. Readers may easily find that when we draw the first return map to \(U^5\), it comes to be resemble to the first return map to \(U^1\). Figure 6 shows the first return map to \(U^n\) that we assumed.\par
\begin{figure}[htbp]
\begin{center}
\includegraphics[width=0.755\textwidth]{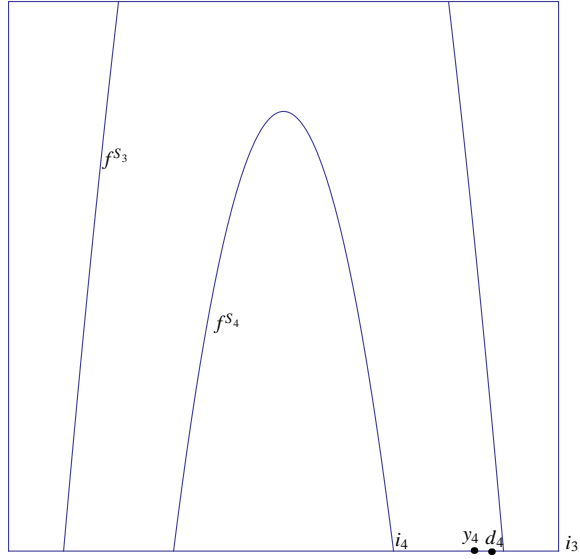}
\caption{First return map into \(U^3\)}
\end{center}
\end{figure}

\begin{figure}[htbp]
\begin{center}
\includegraphics[width=0.755\textwidth]{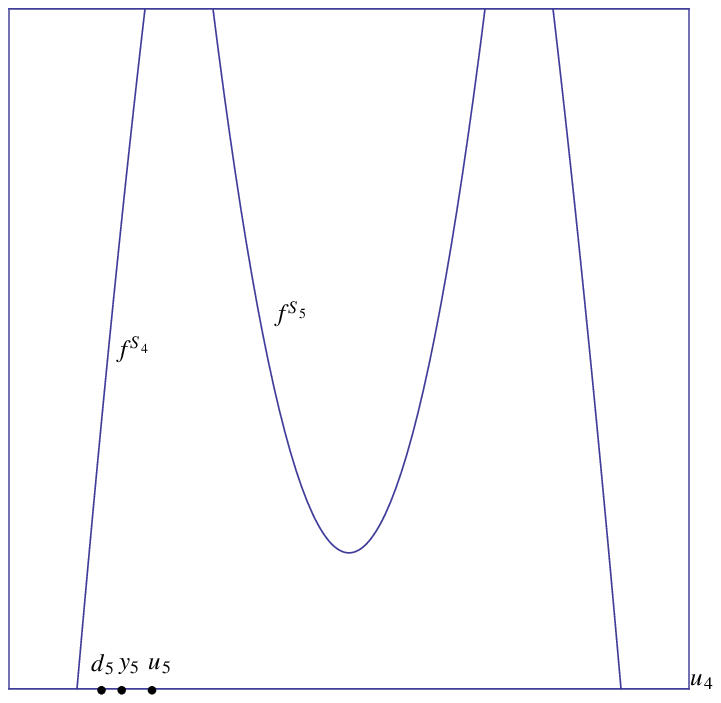}
\caption{First return map into \(U^4\)}
\end{center}
\end{figure}

\begin{figure}
\centering
\subfigure[\(n=4k+1\)]{
\includegraphics[width=2.0in]{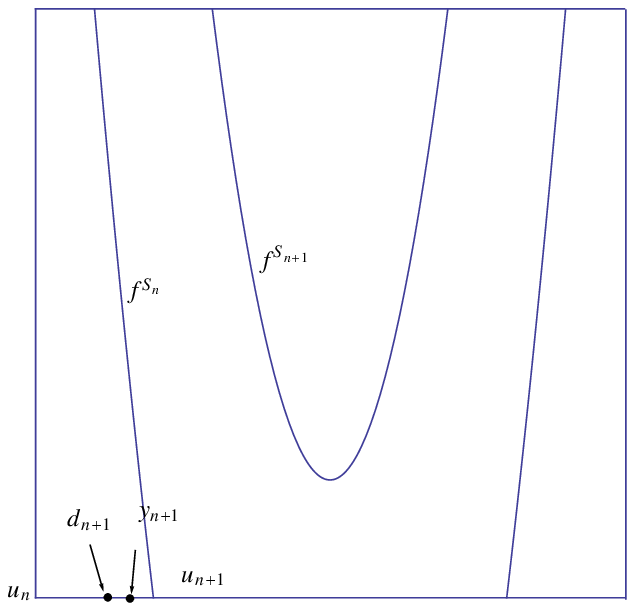}}
\subfigure[\(n=4k+2\)]{
\includegraphics[width=2.0in]{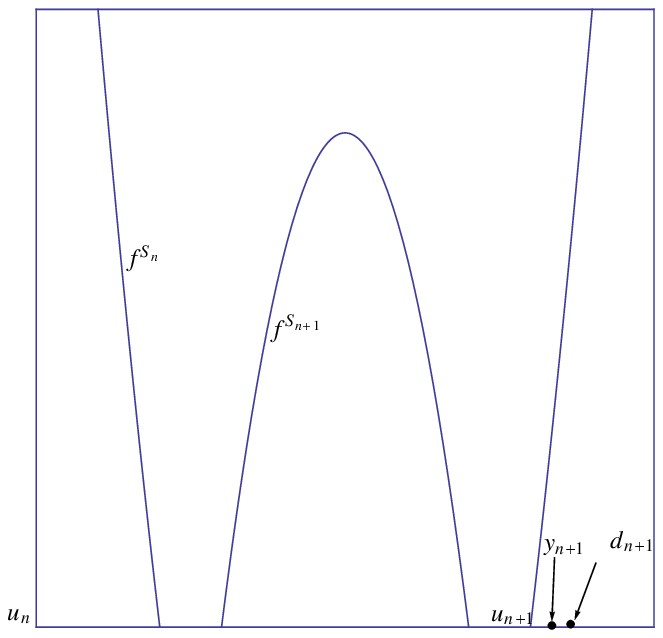}}
\subfigure[\(n=4k+3\)]{
\includegraphics[width=2.0in]{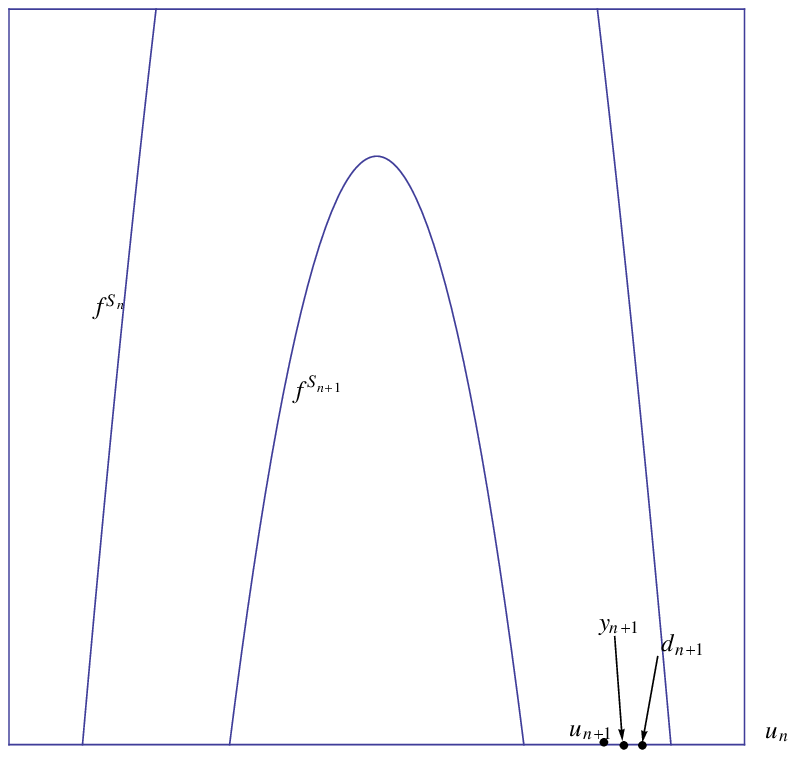}}
\subfigure[\(n=4k+4\)]{
\includegraphics[width=2.0in]{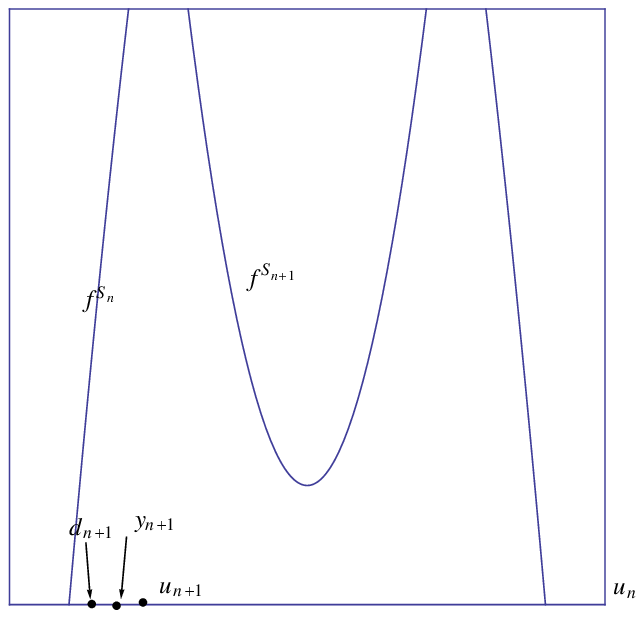}}
\caption{First return map into \(U^n\)}
\end{figure}

Let us prove Proposition 1 by induction:\par
Suppose the first return map into \(U^n\) illustrated in Figure 6 is true for \(4k+1\), \(4k+2\), \(4k+3\), \(4k+4\), then for \(n=4k+5\), we have the following:\newline
\(f^{S_{4k+5}}\) is monotone in \((u_{4k+5},c)\), and \(f^{S_{4k+6}}\) is monotone in \((u_{4k+6},c)\),\newline \(\left\{
\begin{array}{c}
 f^{S_{4k+5}}(u_{4k+5})=u_{4k+3} \\
 f^{S_{4k+5}}(u_{4k+6})=u_{4k+5}
\end{array}
\right.\), \(\left\{
\begin{array}{c}
 f^{S_{4k+6}}(u_{4k+6})=u_{4k+4} \\
 f^{S_{4k+6}}(c)=d_{4k+6}\in U^{4k+5}\text{ but }d_{4k+6}\notin U^{4k+6}
\end{array}\right.\),\newline
 \(\left\{
\begin{array}{c}
 f^{S_{4k+5}}(d_{4k+6})=d_{4k+7}\in U^{4k+5} \\
 f^{S_{4k+5}}(y_{4k+6})=d_{4k+9}\in U^{4k+5}
\end{array}
\right.\).\newline
So \((d_{4k+6},y_{4k+6})\) is contained in one component of \(U^{4k+5}\cap D_{U^{4k+5}}\).  It shows that the first return map into \(U^{4k+5}\) is coincide with Figure 6. By the same procedure, it can be proved that Figure 6 is correct for \(4k+6\), \(4k+7\),\(4k+8\). Thus by the induction axiom, Figure 6 is valid for every \(n\).\par
Let us recall Corollary 1, as \(c_{S_n+\text{...}}\in (d_n,y_n)\), we can observed that \(orb^+(c)\subset \underset{n}{\cup }(d_n,y_n)\). So from Figure 6, there are specifically two components of \(U^n\cap D_{U^n}\) that intersect with \(orb^+(c)\), one is the central branch that contains \(c\), the other is the branch contains \((d_{n+1},y_{n+1})\). Thus we conclude our proof of Proposition 1.\newline
\end{proof}

\begin{figure}[htbp]
\begin{center}
\includegraphics[width=0.755\textwidth]{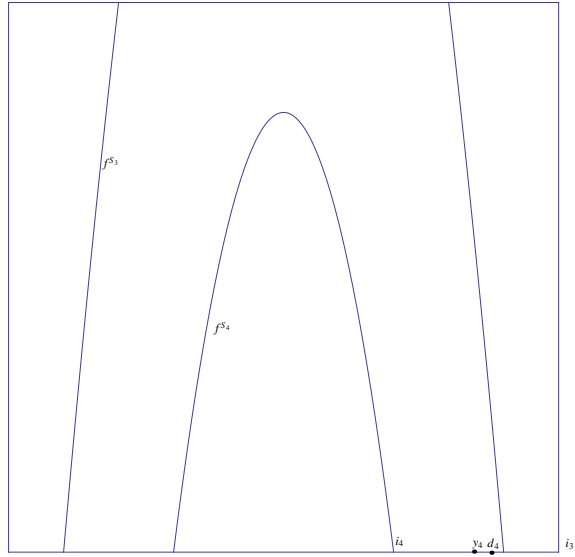}
\caption{First return map into \(I^3\)}
\end{center}
\end{figure}

Next, we are going to discuss the relationship between \(I^n\) and \(U^n\). The first time when \(I^n\) and \(U^n\) differs is that \(n=3\). As \(\left\{
\begin{array}{c}
 f^{S_3}(i_3)=i_2=u_2 \\
 f^{S_3}(u_3)=u_1 \\
 f^{S_3}(d_4)=d_5
\end{array}
\right.\), so \(i_3\in (d_4,u_3)\)(as illustrated in Figure 7). Suppose \(i_{n-1}\in (d_n,u_{n-1})\) or \(i_{n-1}\in (\overset{\frown }{d}_n,u_{n-1})\), as \(\left\{
\begin{array}{c}
 f^{S_n}(i_n)=i_{n-1} \\
 f^{S_n}(u_n)=u_{n-2} \\
 f^{S_n}(d_{n+1})=d_{n+2}
\end{array}
\right.\) and \(f^{S_n}\) is monotone in \((u_n,c)\) by Corollary 2, so \(i_n\in (d_{n+1},u_n)\) or \(i_n\in (\overset{\frown }{d}_{n+1},u_n)\). The fact is true for every \(n\) by the induction axiom. Moreover, as \(\left\{
\begin{array}{c}
 f^{S_n}(y_{n+1})=y_{n+4}\in I^n \\
 f^{S_n}(d_{n+1})=d_{n+2}\in I^n
\end{array}
\right.\), we can infer that \((y_{n+1},d_{n+1})\in I^n\). As a result of Corollary 1, there are specifically two components of \(I^n\cap D_{I^n}\) that intersect with \(orb^+(c)\), one is the central branch that contains \(c\), the other is the branch contains \((d_{n+1},y_{n+1})\). Thus completes the proof of Theorem 1.
\end{proof}

We are going to prove the property of Fibonacci map we just get in Theorem 1 equals the original definition of the Fibonacci map with some additional requirements. \newline
\begin{theorem}
 If \(f:[0,1]\rightarrow [0,1]\) is a unimodal map. \(I^1=(\overset{\frown }{q},q)\), \(\left\{I^k\right\}\) is the principal nest starting from \(I^1\). \(I^k\cap D_{I^k}\) have specifically two components that intersect with \(orb^+(c)\), donated by \(I_0^k\) and \(I_1^k\) in which \(c\in I_0^k\). And \(f\) satisfies \(R_{I^k}|_{I_0^k}=f^{S_{k+1}}\), \(R_{I^k}|_{I_1^k}=f^{S_k}\). Furthermore, \(f^{S_k}(c)\notin I^k\) for every \(k\), then \(f\) is a Fibonacci map.
 \end{theorem}

\begin{proof}
 Note that \(c_{S_{k+1}}\notin I^{k+1}=I_0^k\) and \(I^k\cap D_{I^k}\) have specifically two components that intersect with \(orb^+(c)\). So \(c_{S_{k+1}}\in I_1^k\). We also have \(c_{S_{k+2}}\in I^{k+1}\), so we obtain \(\left|c_{S_{k+1}}\right|>\left|c_{S_{k+2}}\right|\). Furthermore \(\left|c_1\right|>\left|c_2\right|\), thus \(f\) is a Fibonacci map(This is another kind of definition of Fibonacci map which equals to our definition, the proof can be found in [1])
 \end{proof}

As a result of  Theorem 1 and Theorem 2, we can describe the Fibonacci map by principal nest. We have:
\begin{theorem}
  If \(f:[0,1]\rightarrow [0,1]\) is a Fibonacci map, \(q\) is the reverse fixed point of \(f\). \(I^1=(\overset{\frown }{q},q)\), \(\left\{I^k\right\}\) is the principal nest of \(f\) starting from \(I^1\). Then \(f\) satisfies the below properties:\newline
  (1) \(I^k\cap D_{I^k}\) have specifically two components that intersect with \(orb^+(c)\), donated by \(I_0^k\) and \(I_1^k\) in which \(I_0^k=I^{k+1}\)\newline
  (2) \(R_{I^k}|_{I_0^k}=f^{S_{k+1}}\), \(R_{I^k}|_{I_1^k}=f^{S_k}\)\newline
  (3) \(f^{S_k}(c)\notin I^k\) for every \(k\)\newline
  In reverse, if a unimodal map \(f:[0,1]\rightarrow [0,1]\) satisfies the above 3 properties, then \(f\) is Fibonacci map.
\end{theorem}

\begin{proof}
If \(f\) is a Fibonacci map, then \(f\) satisfies (1) and (2) by Theorem 1. Also we have \(i_n\in (d_{n+1},u_n)\) or \(i_n\in (\overset{\frown }{d}_{n+1},u_n)\) in the proof of Theorem 1 combined with the fact that \(\left|u_n\right|<\left|d_n\right|\) (see Lemma 2), so \(f\) satisfies (3). And the the opposite direction is valid due to Theorem 2.
\end{proof}

 \section{References}
[1] M. Lyubich and J. Milnor, \textit{The unimodal Fibonacci map}, J. Amer. Math. Soc. 6,425-457,(1993).\newline\newline
[2] H. Bruin, G. Keller, T. Nowicki, S. van Strien, \textit{Wild Cantor attractors exist}, Ann. of Math. 143, 97-130,(1996).\newline\newline
[3] H. Bruin, \textit{Combinatorics of (Fibonacci-like) unimodal maps}, Notes used for Spring School, (2006).\newline\newline
[4] S. Li, W. Shen, \textit{Hausdorff dimension of Cantor attractors in one-dimensional dynamics}, Invent. Math. 171, 345-387,(2008).

    \end{CJK*}
    \end{document}